\documentclass[a4paper,11pt]{amsart}
\usepackage[latin1]{inputenc}
\usepackage[T1]{fontenc}
\usepackage{amsmath}
\usepackage{amsfonts}
\usepackage{amssymb}
\usepackage{amsthm}
\usepackage{color}
\usepackage{amsthm}
\usepackage{hyperref}
\usepackage{graphicx}
\usepackage{mathrsfs, stmaryrd, mathtools}
\usepackage{cleveref}
\newtheorem{thm}{Theorem}[section]
\newtheorem{defn}[thm]{Definition}
\newtheorem{rmk}[thm]{Remark}
\newtheorem{prop}[thm]{Proposition}
\newtheorem{lem}[thm]{Lemma}
\newtheorem{cor}[thm]{Corollary}

\DeclarePairedDelimiterX\intff[2]{[}{]}{#1,#2}
\DeclarePairedDelimiterX\intfo[2]{[}{)}{#1,#2}
\DeclarePairedDelimiterX\intof[2]{(}{]}{#1,#2}
\DeclarePairedDelimiterX\intoo[2]{(}{)}{#1,#2}
\DeclarePairedDelimiter{\pars}{(}{)}
\DeclarePairedDelimiter{\absolute}{|}{|}
\DeclarePairedDelimiter{\bracks}{[}{]}
\DeclarePairedDelimiter{\braces}{\lbrace}{\rbrace}
\DeclarePairedDelimiterX{\setof}[2]{\lbrace}{\rbrace}{#1\,{:}\,#2}
\DeclarePairedDelimiterX{\bracksof}[2]{[}{]}{#1\,\delimsize\vert\,#2}
\DeclarePairedDelimiterX{\parsof}[2]{(}{)}{#1\,\delimsize\vert\,#2}
\DeclarePairedDelimiterXPP\lnorm[2]{}\lVert\rVert{_{#1}}{#2}

\newcommand\gnP{\mathbf P}
\newcommand\gnE{\mathbf E}
\newcommand\gnVar{\mathtt{Var}}
\newcommand\gnRV{X}
\newcommand\gnT{{T}}  
\newcommand\gnNd{u}              
\newcommand\gnNdd{v}

\newcommand\gnCh[1]{k_{#1}}      
\newcommand\gnZ[1]{Z_{#1}}       
\newcommand\gnH{H}               
\newcommand\gnTrees{\mathcal T}  
\newcommand\gnsT{T}

\newcommand\gnV[1]{V_{#1}}
\newcommand\gnX[1]{X_{#1}}

\newcommand\subtree[2]{#1\bracks{#2}}     
\newcommand\ancestor[1]{\phi_{#1}}        

\newcommand\prune[1]{\mathtt{pru}_{#1}}   
\newcommand\cut[1]{\mathtt{cut}_{#1}}

\newcommand\gwCh{\mu}                     
\newcommand\mean{m} 
\newcommand\gwP{\mathbf P}
\newcommand\gwE{\mathbf E}
\newcommand\gwTr[2]{P_{#1}\pars*{#2}}     
\newcommand\gwf{f}                        
\newcommand\gwq{q}                        
\newcommand\gwgamma{\gamma}               

\newcommand\pruTrees[1]{\mathcal T_{#1}}  
\newcommand\pruP[1]{\mathbf P^{pru}_{#1}}

\newcommand\stP[1]{\mathbf P^{st}_{#1}}

\newcommand\brwD{\theta}
\newcommand\bgwP{\mathbf P}

\newcommand\bstP[1]{\mathbf P^{st}_{#1}}
\newcommand\bstE[1]{\mathbf E^{st}_{#1}}

\newcommand\range[1]{R_{#1}}
\newcommand\rangeS[1]{S_{#1}}
\newcommand\rangeG[1]{G_{#1}}
\newcommand\rangeR{F}
\newcommand\gap[2]{g_{#1}^{#2}}

\author[]{Tianyi Bai, Pierre Rousselin}

\thanks{\textsc{Laboratoire de G\'eom\'etrie, Analyse et Applications, Universit\'e Sorbonne Paris Nord, CNRS UMR 7539, Villetaneuse, France.}}

\thanks{\emph{E-mail:} \texttt{bai@math.univ-paris13.fr}, \texttt{rousselin@math.univ-paris13.fr}}

\title[]{Branching random walks conditioned on rarely survival}
\date{}

\begin{document}
\maketitle
\begin{abstract}In this paper, we show that a Galton-Watson tree conditioned to have a fixed number of particles in generation $n$ converges in distribution as $n\rightarrow\infty$, and with this tool we study the span and gap statistics of a branching random walk on such trees, which is the discrete version of Ramola, Majumdar and Schehr \cite{phycon2}, generalized to arbitrary offspring and displacement distributions with moment constraints.
\end{abstract}
\bibliographystyle{plain}

\section{Introduction}
Consider a Galton-Watson tree $\gnT$ with offspring distribution $\gwCh$ and regularity conditions
\begin{equation}\label{eq:assumption_gw}
\begin{gathered}
\gwCh(0),\gwCh(1)>0,\,\gwCh(0)+\gwCh(1)<1,\\
\mean:=\sum_{k=1}^\infty k\gwCh(k)\in(0,\infty),\,\sum_{k=1}^\infty k^2\gwCh(k)<\infty.
\end{gathered}
\end{equation}
We denote by $\gnZ{n}$ the population of generation $n$, 
\begin{figure}[ht]
\centering
\includegraphics[scale=0.6]{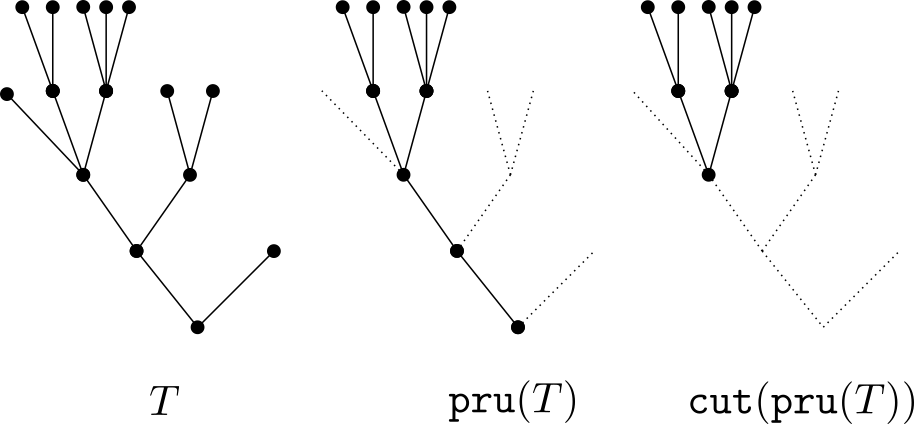}
\caption{Essential structure $\cut{n}(\prune{n}(\gnT))$ for $n=4,k=5$.}
\label{fig1}
\end{figure}
and by $\cut{n}(\prune{n}(\gnT))$ the reduced tree formed by the family tree of nodes in generation $n$ up to their youngest common ancestor, as illustrated in \Cref{fig1}. Then our first result is that
\begin{thm}\label{thm:gw}
Fix $k\ge 1$.
Under \eqref{eq:assumption_gw}, one can construct an explicit probability measure $\stP{k}$ (see \Cref{pp:st}) such that, as $n\rightarrow\infty$, for any set $B$ of finite trees,
\[
\gwP\parsof*{\cut{n}(\prune{n}(\gnT))\in B}{\gnZ{n}=k}\rightarrow\stP{k}(B).
\]
\end{thm}

Moreover, consider the branching random walk $(\gnV{\gnNd})_{\gnNd\in\gnT}$ indexed by a Galton-Watson tree: 
\[\gnV{\gnNd}=\sum_{\varnothing<\gnNdd\le\gnNd}\gnX{\gnNdd},\,\gnV{\varnothing}=0,\]
where $\gnX{\gnNdd}\overset{\text{i.i.d.}}{\sim}\brwD$ for all $\gnNdd\in\gnT\backslash\{\varnothing\}$, and $\brwD$ is a distribution on $\mathbb R$ with regularity conditions 
\begin{equation}\label{eq:assumption_brw}
\begin{aligned}
\gnE[\gnRV]=0,\,\gnVar(\gnRV)=1,\,
\gnE[\exp(t\gnRV)]<\infty,\forall t\in\mathbb R\text{, where }X\sim\brwD.
\end{aligned}
\end{equation}

\begin{figure}[ht]
\centering
\includegraphics[scale=0.6]{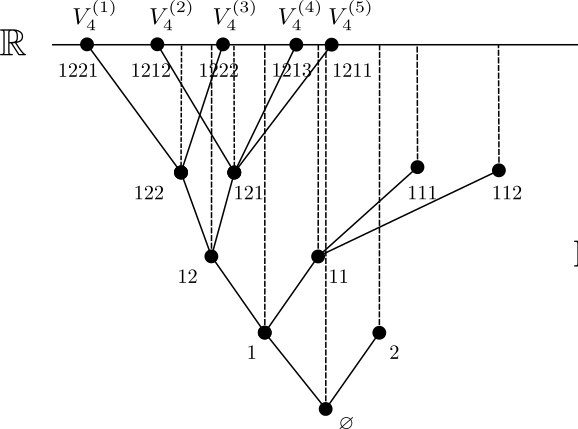}
\caption{Spatial positions $(\gnV{n}^{(i)})$ illustrated with $n=4,k=5$,}
\label{fig2}
\end{figure}

We list the positions of nodes in generation $n$ in increasing order,
\[\gnV{n}^{(1)}\le\cdots\le\gnV{n}^{(\gnZ{n})},\]
as showed in \Cref{fig2},
then we study its {\bf span}
\[
\range{n}:=\gnV{n}^{(\gnZ{n})}-\gnV{n}^{(1)}
\]
and its successive {\bf gaps}
\[
\gap{n}{i}:=\gnV{n}^{(i+1)}-\gnV{n}^{(i)},\,1\le i\le \gnZ{n}-1.
\]
For simplicity, we write $\range{}$ and $(\gap{}{i})$ for the span and gaps of the last generation for a finite tree.
Recall that $\mean=\sum_{k=1}^\infty k\gwCh(k)$ is the expected number of children, and our second result is then
\begin{thm}\label{thm:brw}
Let $k\ge 2$ and $1\le i\le k-1$. If \eqref{eq:assumption_gw} and \eqref{eq:assumption_brw} are satisfied, then the law of
$\range{n}$ and $(\gap{n}{i})$ under $\gwP\parsof{\cdot}{\gnZ{n}=k}$ converges in distribution to that of $\range{}$ and $(\gap{}{i})$ under $\stP{k}(\cdot)$, and there are explicit constants $C_1,C_2,C_3$ (see \Cref{lm:rangeS}, \Cref{pp:gap}) such that, as $x\rightarrow\infty$,
\[
\stP{k}\pars*{\range{}>x}=
\begin{cases}
(C_1+o(1))x^{-2}, &\mean=1,\\
\exp(-(C_2+o(1))x),&\mean\ne 1,
\end{cases}
\]
\[
\stP{k}\pars*{\gap{}{i}>x}=
\begin{cases}
(C_1C_3+o(1))x^{-2}, &\mean=1,\\
\exp(-(C_2+o(1))x),&\mean\ne 1.
\end{cases}
\]
\end{thm}
\begin{rmk}
We use the unified assumptions \eqref{eq:assumption_gw} and \eqref{eq:assumption_brw} for convenience. These assumptions can be further refined for specific cases: 
\begin{itemize}
\item
In the critical case $\mean=1$, $\brwD$ only need to have finite $(2+\delta)$-th moment (for any $\delta>0$) instead of exponential moments. 
\item 
In the supercritical case $\mean>1$, we only need the $L\log L$ condition, $\sum_{i=1}^\infty i\log i\gwCh(i)<\infty$, instead of the finite variance condition in \eqref{eq:assumption_gw}.
\item
In the subcritical case $\mean<1$, we only need the $L\log L$ condition for \Cref{chapter:smalltree} except in Part 2, \Cref{cor:st_height}.
\end{itemize}
\end{rmk}

This paper is mainly motivated by \cite{phycon} and \cite{phycon2},
where Ramola, Majumdar and Schehr studied the span and gaps for the branching Brownian motion via a PDE method. Their result corresponds to the continuous version of \Cref{thm:brw} with geometric $\gwCh$ and Gaussian $\brwD$. In particular, we show that the asymptotic for gap statistics are no longer independent of $k$ and $i$ for the critical case with non-geometric offspring distribution, see \Cref{rmk:phy}.

The study of the reduced Galton-Watson tree $\prune{n}(\gnT)$ at least dates back to Fleischmann and Prehn \cite{fleischmann1974grenzwertsatz}, Fleischmann and Siegmund-Schultze \cite{zbMATH03469834}, where it is showed that the limit for the critical case $\mean=1$ is the Yule tree. See also Curien and Le Gall \cite{zbMATH06696268} for properties and applications of the Yule tree.
In particular, for the critical Galton Watson tree conditioned on non-extinction, one has $\gnZ{n}=\Theta(n)$, and the conditioned case $\{0<\gnZ{n}\le\phi(n)\}$, where $\phi(n)=O(n),\phi(n)\rightarrow\infty$ is recently studied in Liu and Vatutin \cite{zbMATH07014723}.

The conditioned limiting behavior of the whole tree $\gnT$ (in contrast to the reduced tree) is known as the local limit. The general result (under the condition $\braces{\gnZ{n}>0}$) for the local limit is the Kesten's tree \cite{zbMATH04028592}, see also Geiger \cite{geiger_1999}. Further, there are detailed discussions for local limits conditioned on rare events ($\{\gnZ{n}<\epsilon\gnE\bracksof{\gnZ{n}}{\gnZ{n}>0}\}$ or $\{\gnZ{n}>\epsilon^{-1}\gnE\bracksof{\gnZ{n}}{\gnZ{n}>0}\}$), see Abraham and Delmas \cite{zbMATH06347442}, \cite{zbMATH06257619}, Abraham, Bouaziz and Delmas \cite{zbMATH07277656}.

In addition, we also establish an extension of the Ratio theorem (cf. Theorem 1.7.4, \cite{bookan}) as an byproduct. Indeed, let 
\[
\gwTr{n}{1,j}:=\gwP\parsof*{\gnZ{n}=j}{\gnZ{0}=1},
\]
denote the transition probabilities for the population of $n$ generations of the Galton-Watson tree, then we have that
\begin{prop}\label{pp:rationplus}
Fix $k\ge 2$. Under \eqref{eq:assumption_gw}, we can construct constants $\gwgamma$ (see Part 2, \Cref{pp:ratio}) and $C_4,C_5$ (see \Cref{rmk:ratio}) such that as $n\rightarrow\infty$,
\[
\frac{\gwTr{n}{1,k}}{\gwTr{n}{1,1}}-\frac{\gwTr{n-1}{1,k}}{\gwTr{n-1}{1,1}}
=
\begin{cases}
(C_4+o(1))n^{-2},        & \mean=1, \\
(C_5+o(1))\gwgamma^{n},  & \mean\ne 1,
\end{cases}
\]
\end{prop}

The paper is organized as follows.
In \Cref{sec:preliminaries} we present systematically the notations and concepts needed.
In \Cref{chapter:smalltree} we study the genealogical properties of $\prune{n}(\gnT)$ and $\cut{n}(\prune{n}(\gnT))$, 
then we prove \Cref{thm:gw} and \Cref{pp:rationplus} in \Cref{pp:st} and  \Cref{rmk:ratio}.
In \Cref{sec:spatial}, we study the span and gaps, and prove \Cref{thm:brw} in \Cref{pp:rangeR} and \Cref{pp:gap}.

\section{Preliminaries}\label{sec:preliminaries}

\subsection{Trees}

A (locally finite, rooted) {\bf planar tree} $\gnT$ is a subset of integer-valued words, 
$\gnT\subset\cup_{n\ge 0}\mathbb N^n_+$, such that:
\begin{itemize}
    \item The {\bf root} $\varnothing\in \gnT$, where by convention we denote $\mathbb N_+^0=\{\varnothing\}$.
    \item If a {\bf node} $\gnNd=(\gnNd_1,\cdots,\gnNd_n)\in \gnT$, 
    then its {\bf parent} $\overleftarrow{\gnNd}:=(\gnNd_1,\cdots,\gnNd_{n-1})\in \gnT$. 
    \item For each node $\gnNd=(\gnNd_1,\cdots,\gnNd_n)\in \gnT$, 
    there exists an integer $\gnCh{\gnNd}(\gnT)\in\mathbb N$ called its {\bf number of children},
    such that for every $j\in\mathbb N, (\gnNd_1,\cdots,\gnNd_n,j)\in \gnT$ 
    if and only if $1\le j \le \gnCh{\gnNd}(\gnT)$.
\end{itemize} 
We only consider locally finite trees, i.e. $\gnCh{\gnNd}(\gnT)<\infty,\,\forall\gnNd\in\gnT$.
We give a few basic notations on trees:
\begin{itemize}
    \item The {\bf set of all planar trees} is denoted by $\gnTrees$.
    \item The {\bf generation/height of a node} is its length as a word, 
    i.e. if $\gnNd=(\gnNd_1,\cdots,\gnNd_n)$, then $\absolute{\gnNd}=n$. 
    The {\bf height of a tree} is then defined as
    $\gnH(\gnT):=\max\setof{\absolute{\gnNd}}{\gnNd\in\gnT}\in\mathbb N\cup\braces{\infty}.$
    \item The {\bf population of generation} $n$ is defined as
    $\gnZ{n}(\gnT):=\#\setof{\gnNd\in\gnT}{\absolute{\gnNd}=n}.$ 
    By construction, $\gnZ{0}(\gnT)=1$ for any tree $\gnT$. 
    \item A node $\gnNd=(\gnNd_1,\cdots,\gnNd_n)\in \gnT$ 
    is an {\bf ancestor} of another one $\gnNdd=(\gnNdd_1,\cdots,\gnNdd_{m})\in\gnT$, 
    denoted by $\gnNd\prec\gnNdd$, if $n<m$ and $\gnNd_i=\gnNdd_i,\,1\le i\le n.$
    The {\bf (youngest) common ancestor} of two nodes $\gnNd,\gnNdd\in\gnT$, 
    denoted by $\gnNd\wedge\gnNdd$, 
    is then the node in $\gnT$ with maximum height, 
    such that $\gnNd\wedge\gnNdd\preceq\gnNd,\gnNdd$. 
    \item For $\gnNd\in\gnT$, the {\bf subtree} rooted at $\gnNd\in\gnNd$ is defined as
    $\subtree{\gnT}{\gnNd}=\setof{\gnNdd\in\mathbb N^n_+}{\gnNd\gnNdd\in\gnT},$
    where $\gnNd\gnNdd$ stands for concatenation of words. 
    It is not hard to check that this set is a tree. 
    In particular, given that $\gnCh{\varnothing}(\gnT)=r$, 
    nodes in the first generations are labeled $1,2,\cdots,r$ by construction, 
    thus subtrees rooted at them are $\subtree{\gnT}{1},\cdots,\subtree{\gnT}{r}.$
\end{itemize}
When there is no confusion for the tree under consideration, we omit the reliance on $\gnT$ and write, for instance, $\gnZ{n}$ for $\gnZ{n}(\gnT)$.

Moreover, given a tree $\gnT$, one can attach the (not necessarily random) spatial structure $(\gnV{\gnNd})_{\gnNd\in\gnT}$:
We attach on each node $\gnNdd\ne\varnothing$ a displacement from its parent $\gnX{\gnNdd}\in\mathbb R$, and set the position of a node $\gnNd$ as
\[
\gnV{\gnNd}=\sum_{\varnothing\prec\gnNdd\preceq\gnNd}\gnX{\gnNdd},\,\gnV{\varnothing}=0.
\]
\if
Moreover, one can then list the positions of nodes in generation $n$, $(\gnV{\gnNd})_{|\gnNd|=n}$, in increasing order,
\[\gnV{n}^{(1)}\le\cdots\le\gnV{n}^{(\gnZ{n})}.\]
In this paper, we are interested in its {\bf span}
\[
\range{n}:=\gnV{n}^{(\gnZ{n})}-\gnV{n}^{(1)}
\]
and its {\bf gaps}
\[
\gap{n}{i}:=\gnV{n}^{(i+1)}-\gnV{n}^{(i)}.
\]
\fi

\subsection{The prune and cut operation}
To study the relative relations of nodes in generation $n$ while omitting irrelevant information,
we define the prune and cut operations on trees (recall the illustration in \Cref{fig1}):
\begin{figure}[ht]
\centering
\includegraphics[scale=0.6]{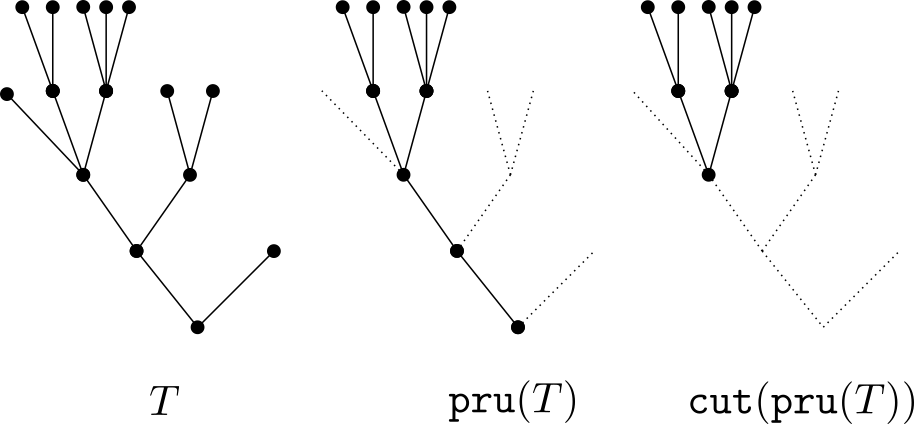}
\end{figure}
\begin{itemize}
    \item For any $\gnT\in\gnTrees$, we construct the {\bf pruned tree} at height $n$ by
    \[
    \prune{n}(\gnT):=
    \setof{\gnNd\in\gnT}{\exists\gnNdd\in\gnT,\absolute{\gnNdd}=n,\gnNd\preceq\gnNdd}.
    \]
    By convention, if $\gnZ{n}(\gnT)=0$, we take $\prune{n}(\gnT)=\braces{\varnothing}$.
    \item Moreover, we define the {\bf cut operation} by
    \[
    \ancestor{n}(\gnT) = \bigwedge_{{\absolute{\gnNd} = n},\,{\gnNd\in\gnT}} \gnNd,\,\text{and then }
    \cut{n}(\gnT)=\subtree{\gnT}{\ancestor{n}(\gnT)}.
    \]
    where by convention, $\ancestor{n}(\gnT)=\varnothing$ if $\gnZ{n}(\gnT)=0.$
\end{itemize}
In addition, we denote the {\bf set of all pruned trees} at height $n$ by
$\pruTrees{n}:=\prune{n}(\gnTrees),$ and  
the set of all pruned trees at height $n$ with $\gnZ{n}=k$ by $\pruTrees{n,k}$. 
In particular, $\pruTrees{n,0}$ contains only one element $\braces{\varnothing}$.
Since trees are assumed to be locally finite, we have that
$\pruTrees{n}=\cup_{k=0}^\infty\pruTrees{n,k}$. Since all these sets are countable, the problem of measurability is trivial.

These operations naturally extend to the branching random walk indexed by these trees, by translation such that the root is always pinned at $0$. Then by construction,
\begin{equation}\label{eq:operation_nature}
\range{n}(\gnsT)=\range{\gnH(\cut{n}(\gnsT))}(\cut{n}(\gnsT))=\range{n}(\prune{n}(\gnsT)),
\end{equation}
and the same thing applies to the gaps.

\subsection{Galton-Watson tree and Ratio theorem}
Let $\gwCh$ be a probability distribution on $\mathbb N$. 
The law of a {\bf Galton-Watson} tree with offspring distribution $\gwCh$
is a probability measure $\gwP$ on the set of planar trees $\gnTrees$, such that for all nodes $\gnNd$,
\[\gnCh{\gnNd}\overset{i.i.d.}{\sim}\gwCh.\]

\if
Throughout this paper we fix an offspring distribution $\gwCh$ such that
\begin{equation}\label{eq:assumption_gw}
\gnVar\gwCh<\infty,\,\gwCh(0),\gwCh(1)>0,\,\gwCh(0)+\gwCh(1)<1.
\end{equation}
These assumptions are essential for fundamental estimates in \Cref{pp:ratio}. 
In addition, we remark that as long as $n\ge k$, we have
$\gwP(\gnZ{n}(\gnT)=k)>0$, so that the conditional probability $\gwP\parsof*{\cdot}{\gnZ{n}(\gnT)=k}$ is well-defined for any $n\ge k$.
\fi

Clearly, the sequence $(\gnZ{n})$ is a Markov chain starting at $Z_0=1$ under $\gwP$, and one can then set its transition probabilities as
\[
\gwTr{n}{i,j}:=\gwP\parsof*{\gnZ{k+n}=j}{\gnZ{k}=i},
\]
where we take $k$ such that $\gwP(\gnZ{k}=i)>0.$ (Under the assumption $\eqref{eq:assumption_gw}$, this is always possible by taking $k$ large enough.)
In particular, 
\[
\gwTr{1}{1,i}=\gwCh(i).
\]

Moreover, we define the {\bf generating function} of this process as
\begin{equation}\label{eq:defn_f}
\gwf(s):=\gwE\pars*{s^{\gnZ{1}(\gnT)}}=\sum_{i=0}^\infty\gwTr{1}{1,i}s^i,
\end{equation}
then its derivatives are
\begin{equation}\label{eq:derivative_f}
\gwf^{(r)}(s)=r!\sum_{\ell\ge r}\binom{l}{r}\gwTr{1}{1,i}s^{\ell-r},
\end{equation}
and its iterations are 
\begin{equation}\label{eq:f_iteration}
\gwf_n(s):=\underbrace{\gwf \circ \gwf \circ \dotsb \circ \gwf}_{\text{$n$ times}}
=\gwE\pars*{s^{\gnZ{n}(\gnT)}}=\sum_{\ell\geq0} \gwTr{n}{1,\ell} s^\ell.
\end{equation}
We also define the {\bf extinction probabilities} as
\begin{equation}\label{eq:defn_q}
\gwq_n:=\gwTr{n}{1,0}=\gwf_n(0),\, \gwq:=\lim_{n\rightarrow\infty}\gwq_n.
\end{equation}
Clearly, $(\gwq_n)$ is a bounded increasing sequence, which guarantees the existence of $\gwq$. 
Moreover, it is standard (Athreya and Ney \cite[Theorem 1.5.1]{bookan}) that $\gwq=1$ if $\mean\le 1$ (except for the trivial case $\gwCh=\delta_1$), and $\gwq<1$ if $\mean>1$, where $\mean:=\sum_{i=1}^\infty i\gwCh(i)$ is the expected number of children.

We finish this section by citing some fundamental estimates that shall be used later,
\begin{prop}[Athreya and Ney {\cite[Section 1.7-1.11]{bookan}}]\label{pp:ratio}
Let $\gwCh$ be an offspring distribution such that
\begin{equation*}
\gwCh(0),\gwCh(1)>0,\,\gwCh(0)+\gwCh(1)<1,\,\mean<\infty.
\end{equation*}
\begin{enumerate}
\item 
There exists a sequence $(\pi_j)$ such that for any $j\ge 1$,
\[
	\lim_{n\rightarrow\infty}\frac{\gwTr{n}{1,j}}{\gwTr{n}{1,1}}\nearrow\pi_j\in(0,\infty),
\]
where $\nearrow$ means non-decreasing limit.
\item 
For any $t\in\mathbb Z$, $i,j,k,l\ge 1$,
\[
	\lim_{n\rightarrow\infty}\frac{\gwTr{n+t}{i,j}}{\gwTr{n}{k,l}}
	=\gwgamma^t\gwq^{i-k}\frac{i\pi_j}{k\pi_l},
\]
where $\gwq$ is the extinction probability in \eqref{eq:defn_q}, and $\gwgamma=\gwf'(\gwq)$. 
\item 
If $\mean=1,\sigma^2:=\sum_{i=1}^\infty i^2\gwCh(i)<\infty$, then for any $i,j\ge 1$,
\[
	\lim_{n\rightarrow\infty} n^2
	\gwTr{n}{i,j}=
	\frac{2i\pi_j}{\sigma^2\sum_{k=1}^\infty \pi_k (\gwCh(0))^k}.
\]
\item 
If $\mean\ne1$, $\sum_{i=1}^\infty i\log i\gwCh(i)<\infty$, then for any $i,j\ge 1$,
\[
    \lim_{n\rightarrow\infty}\gwgamma^{-n}\gwTr{n}{i,j}=i\gwq^{i-1}v_j,
\]
where $(v_j)$ is determined by $Q(s)=\sum_{j=0}^\infty v_js^j,0\le s<1$, 
with $Q$ the unique solution of
\[
Q(\gwf(s))=\gwgamma Q(s)(0\le s<1),\,Q(q)=0,\,\lim_{s\rightarrow \gwq}Q'(s)=1.
\]
\end{enumerate}
\end{prop}

\subsection{Branching random walk and Cramer's theorem}
On a Galton-Watson tree $\gnT$, we shall consider the {\bf branching random walk} $(\gnV{\gnNd})_{\gnNd\in\gnT}$
by attaching i.i.d. spatial displacements
\[\gnX{\gnNd}\overset{i.i.d.}{\sim}\brwD,\,\forall \gnNd\in\gnT\backslash\{\varnothing\},\]
whose probability measure is still denoted by $\gwP$ for simplicity. For other probability measures on trees that we shall construct, we also abuse the same notations for the corresponding spatial process.

Moreover, we cite a fundamental estimate for random walks, 
\begin{lem}\label{lm:cremer}
Let $\brwD$ be a distribution on $\mathbb R$, and let $(\gnRV_i)$ be i.i.d. random variables distributed as $\brwD$. Assume that $\gnE[\gnRV_1]=0$, $\gnVar(\gnRV_1)=1$.
\begin{enumerate}
\item \cite[Theorem 5.7]{petrov1995limit}
If $\gnE\bracks*{|\gnRV_1|^{2+\delta}}<\infty$, then
there exists $C_1,C_2,C_3>0$ such that for any $n\ge 1,x>0$,
\[
	\gnP\pars*{\sum_{i=1}^n \gnRV_i>x}\le \frac{C_1n}{x^3}+e^{-C_2\frac{x^2}{n}},
\]
\[
	\absolute*{
			\gnP\pars*{\sum_{i=1}^n \gnRV_i\le x} 
			- \Phi\pars*{\frac{x}{\sqrt n}}
		}
	\le \frac{C_3}{\sqrt n},
\]
where $\Phi(x)$ is the cumulative distribution function of the standard Gaussian distribution $\mathcal{N}(0,1)$.
\item (Cramer's theorem)
If
\[
\Lambda(t):=\log\gnE[\exp(t\gnRV_1)]<\infty,\,\forall t\in\mathbb R,
\]
then
\[
\lim_{n\rightarrow\infty}\frac{1}{n}\log
		\pars*{\gnP\pars*{\sum_{i=1}^n \gnRV_i\ge nx}}
=-\sup_{t\in\mathbb R}(tx-\Lambda(t)).
\]
\end{enumerate}
\end{lem}

\section{Galton-Watson trees conditioned on rarely survival}\label{chapter:smalltree}
\subsection{Pruned Galton-Watson trees}
We first study the distribution of $\prune{n}(\gnT)$. Recall that $\pruTrees{n,k}$ stands for the set of all pruned trees with $k$ nodes in generation $n$.
\begin{defn}
For any pair of integers $(n,k)$ such that $\gwTr{n}{1,k}>0$, 
we denote by $\pruP{n,k}$ the law of $\prune{n}(\gnT)$, supported on $\pruTrees{n,k}$, 
where $\gnT$ is sampled under the law of $\gwP\parsof*{\cdot}{\gnZ{n}=k}$.
In other words, for any $A\subset\pruTrees{n,k}$,
\begin{equation}\label{eq:defnPGW}
\begin{aligned}
\pruP{n,k}(A)&=\gwP\parsof*{\prune{n}(\gnT)\in A}{\gnZ{n}=k}=\frac{\gwP\pars*{\prune{n}(\gnT)\in A}}{\gwTr{n}{1,k}}.
\end{aligned}
\end{equation}
\end{defn}
Recall that $\subtree{\gnT}{i}$ denotes the $i$-th subtree of $\gnT$ rooted at the first generation, then
\begin{prop}\label{prop:condi_is_multi}
Under \eqref{eq:assumption_gw}, for any pair of integers $(n,k)$ such that $\gwTr{n}{1,k}>0$, 
let $r\ge 1$ and let $k_1,\cdots,k_r$ be positive integers such that
$\sum_{i=1}^r k_i=k.$
Then for any $A_i\subseteq\pruTrees{n-1,k_i},\,1\le i\le r$, 
\begin{align*}
&\pruP{n,k}\pars*{\gnZ{1}(T)=r,\subtree{\gnT}{i}\in A_i, 1\le i\le r}\\
=&
\frac{1}{\gwTr{n}{1, k}}
\frac{\gwf^{(r)}(\gwq_{n-1})}{r!}
\prod_{i=1}^r \gwTr{n-1}{1,k_i}
\prod_{i=1}^r \pruP{n-1,k_i}(A_i),
\end{align*}
where $\gwf$ and $\gwq_{n-1}$ are defined in \eqref{eq:defn_f}, \eqref{eq:defn_q}.
\end{prop}
\begin{proof}
By \eqref{eq:defnPGW}, 
\begin{align*}
 &\gwTr{n}{1,k}\pruP{n,k}\pars*{\gnZ{1}(T)=r,\subtree{\gnT}{i}\in A_i, 1\le i\le r}\\
=&\gwP\pars*{\gnZ{1}\pars*{\prune{n}(\gnT)}=r,(\prune{n}(\gnT))[i]\in A_i,1\le i\le r}\\
=&\sum_{\ell\ge r}\sum_{1\le j_1<\cdots<j_r\le \ell}\gwP(\gnZ{1}(\gnT)=\ell;
\prune{n-1}(\subtree{\gnT}{j_i})\in A_i,i=1,\cdots,r;\\
&\qquad\qquad\qquad\qquad\qquad\qquad\qquad\qquad
\gnZ{n-1}(\subtree{\gnT}{i})=0, i\not\in\braces{j_1,\cdots,j_r}).
\end{align*}

Let $B_i=\setof*{\gnT\in\gnTrees}{\prune{n-1}{\gnT}\in A_i}$, 
by the self-similarity of Galton-Watson trees, 
the equation above can be further simplified as
\begin{equation}\label{eq:inPrune}
\begin{aligned}
 &\gwTr{n}{1,k}\pruP{n,k}\pars*{\gnZ{1}(T)=r,\subtree{\gnT}{i}\in A_i, 1\le i\le r}\\
=&\sum_{\ell\ge r}\sum_{1\le j_1<\cdots<j_r\le \ell}
\gwTr{1}{1,\ell}\gwq_{n-1}^{\ell-r}\prod_{i=1}^r\gwP(B_i).
\end{aligned}
\end{equation}
Moreover, by \eqref{eq:defnPGW}, we have that
\[
\gwP(B_i)=\gwTr{n-1}{1,k_i}\pruP{n-1,k_i}(A_i),
\]
put it in \eqref{eq:inPrune}, and it suffices to show that
\[
\sum_{\ell\ge r}\sum_{1\le j_1<\cdots<j_r\le \ell}
\gwTr{1}{1,\ell}\gwq_{n-1}^{\ell-r}=
\frac{\gwf^{(r)}(\gwq_{n-1})}{r!}.
\]
Indeed, we have
\begin{align*}
&\sum_{\ell\ge r}\sum_{1\le j_1<\cdots<j_r\le \ell}
\gwTr{1}{1,\ell}\gwq_{n-1}^{\ell-r}\\
=&\sum_{\ell\ge r}\binom{\ell}{r}\gwTr{1}{1,\ell}\gwq_{n-1}^{\ell-r}
=\frac{\gwf^{(r)}(\gwq_{n-1})}{r!},
\end{align*}
where the second line follows from \eqref{eq:derivative_f}.
\end{proof}

Moreover, we give two more properties of $\pruP{n,k}$: 
\begin{cor}\label{cor:pruHeight}
Under \eqref{eq:assumption_gw}, for any $1\le u\le n$,
any $A\subseteq\pruTrees{u,k}$,
\begin{align*}
&\pruP{n,k}(\gnZ{1}(\gnT)=\cdots=\gnZ{n-u}(\gnT)=1,\subtree{\gnT}{\underbrace{11\dotsb1}_{\text{$n-u$ times}}}\in A)\\
=&\frac{\gwTr{n}{1,1}}{\gwTr{n}{1,k}}
\frac{\gwTr{u}{1,k}}{\gwTr{u}{1,1}}
\pruP{u,k}(A),
\end{align*}
where $\underbrace{11\dotsb1}_{\text{$n-u$ times}}$ means the first node (and also the only node, under the condition $\gnZ{1}(\gnT)=\cdots=\gnZ{n-u}(\gnT)=1$) in generation $n-i$.
\end{cor}
\begin{proof}
Take $r=1$ in \Cref{prop:condi_is_multi}, we have that
\begin{align*}
\pruP{n,k}\pars*{\gnZ{1}(T)=1,\subtree{\gnT}{1}\in B}
&=\frac{\gwTr{n-1}{1,k}}{\gwTr{n}{1, k}}\gwf'(\gwq_{n-1})\pruP{n-1,k}(B)\\
&=\frac{\gwTr{n-1}{1,k}}{\gwTr{n}{1,k}}\frac{\gwTr{n}{1,1}}{\gwTr{n-1}{1,1}}\pruP{n-1,k}(B),
\end{align*}
for any $B\subseteq\pruTrees{n-1,k}$,
where we use \eqref{eq:f_iteration} and \eqref{eq:defn_q} to deduce that 
$\gwf'(\gwq_{n-1})=\frac{\gwf'_n(0)}{\gwf'_{n-1}(0)}=\frac{\gwTr{n}{1,1}}{\gwTr{n-1}{1,1}}.$
The result follows by using this relation $n-u$ times inductively.
\end{proof}
\begin{cor}\label{cor:pru_r}
Under \eqref{eq:assumption_gw}, for any $r,k\ge 2$, if $\gwf^{(r)}(\gwq)<\infty$, then
\begin{equation*}
\begin{aligned}
\lim_{n\rightarrow\infty}\frac{\pruP{n,k}(\gnZ{1}(\gnT)= r)}{\gwTr{n}{1, 1}^{r-1}}
=\gwgamma^{-r} \frac{\gwf^{(r)}(\gwq)}{r!}
\sum_{\substack{k_1, k_2, \cdots, k_r \ge 1\\ k_1 + \dotsb + k_r = k}}
\frac{\pi_{k_1} \dotsb \pi_{k_r}}{\pi_k},
\end{aligned}
\end{equation*}
\begin{equation*}
\begin{aligned}
\lim_{n\rightarrow\infty}\frac{\pruP{n,k}(\gnZ{1}(\gnT)\ge r)}{\gwTr{n}{1, 1}^{r-1}}
=\gwgamma^{-r} \frac{\gwf^{(r)}(\gwq)}{r!}
\sum_{\substack{k_1, k_2, \cdots, k_r \ge 1\\ k_1 + \dotsb + k_r = k}}
\frac{\pi_{k_1} \dotsb \pi_{k_r}}{\pi_k},
\end{aligned}
\end{equation*}
where $\gwf$ is defined in \eqref{eq:defn_f}, $\gwq,\gwgamma$ and $(\pi_k)$ are defined in \Cref{pp:ratio}.
\end{cor}
\begin{proof}
For the first equation, by \Cref{prop:condi_is_multi}, we take $A_i=\pruTrees{n-1,k_i}$ and sum over all choices of $(k_i)$, then
\[
\pruP{n,k}(\gnZ{1}(\gnT)= r)=
\frac{1}{\gwTr{n}{1, k}}
\frac{\gwf^{(r)}(\gwq_{n-1})}{r!}
\sum_{\substack{k_1, k_2, \cdots, k_r \ge 1\\ k_1 + \dotsb + k_r = k}}
\prod_{i=1}^r \gwTr{n-1}{1,k_i}.
\]
Divide both sides by $\gwTr{n}{1,1}^{r-1}$, we have that
\[
\frac{\pruP{n,k}(\gnZ{1}(\gnT)= r)}{\gwTr{n}{1, 1}^{r-1}}=
\frac{\gwTr{n}{1, 1}}{\gwTr{n}{1, k}}
\frac{\gwf^{(r)}(\gwq_{n-1})}{r!}
\sum_{\substack{k_1, k_2, \cdots, k_r \ge 1\\ k_1 + \dotsb + k_r = k}}
\prod_{i=1}^r \frac{\gwTr{n-1}{1,k_i}}{\gwTr{n}{1, 1}}.
\]
For fixed $r$ and $k$, there are only finitely many terms on the right hand side, thus we can take the limit separately for each fraction by \Cref{pp:ratio}, and the result follows.

To deal with $\pruP{n,k}(\gnZ{1}(\gnT)\ge r)$, we sum over 
\[
\pruP{n,k}\pars*{\gnZ{1}(T)\ge r,\gnZ{n-1}(\subtree{\gnT}{i})=k_i, 1\le i\le r-1}
\]
for all $\sum_{i=1}^r k_i=k$. In other words, we consider the first $r-1$ subtrees to give $k_i$ offspring each, while the rest subtrees give $k_r$ offspring in total. 
By the proof of \Cref{prop:condi_is_multi}, we have that
\begin{align*}
&\gwTr{n}{1,k}\pruP{n,k}\pars*{\gnZ{1}(T)\ge r,\gnZ{n-1}(\subtree{\gnT}{i})=k_i, 1\le i\le r-1}\\
=&\sum_{\ell\ge r}\sum_{1\le j_1<\cdots<j_{r-1}< \ell}\gwP\bigg(
\gnZ{1}(\gnT)=\ell;
\gnZ{n-1}(\subtree{\gnT}{j_i})=k_i,i=1,\cdots,r-1;\\
&\qquad\qquad\qquad
\sum_{\substack{i<j_{r-1}\\i\not\in\braces{j_1,\cdots,j_{r-1}}}}\gnZ{n-1}(\subtree{\gnT}{i})=0;
\sum_{i>j_{r-1}}\gnZ{n-1}(\subtree{\gnT}{i})=k_r\bigg)\\
=&\sum_{\ell\ge r}\sum_{1\le j_1<\cdots<j_{r-1}< \ell}
\gwTr{1}{1,\ell}\prod_{i=1}^{r-1}\gwTr{n-1}{1,k_i}\cdot\gwq_{n-1}^{j_{r-1}-r+1}\gwTr{n-1}{\ell-j_{r-1},k_r}.
\end{align*}
Write $j_{r-1}=j$ for short, and we can simplify this term into
\begin{align*}&\gwTr{n}{1,k}\pruP{n,k}\pars*{\gnZ{1}(T)\ge r,\gnZ{n-1}(\subtree{\gnT}{i})=k_i, 1\le i\le r-1}\\
=&\gwTr{1}{1,\ell}\prod_{i=1}^{r-1}\gwTr{n-1}{1,k_i}
\sum_{\ell=r}^\infty
\sum_{j=r-1}^{\ell-1}\binom{j-1}{r-2}
\gwq_{n-1}^{j-r+1}\gwTr{n-1}{\ell-j,k_r}.
\end{align*}
Divide by $(\gwTr{n}{1,1})^r$, then
\begin{equation}\label{eq:pgw2}
\begin{aligned}
&\frac{\gwTr{n}{1,k}}{(\gwTr{n}{1,1})^r}\pruP{n,k}\pars*{\gnZ{1}(T)\ge r,\gnZ{n-1}(\subtree{\gnT}{i})=k_i, 1\le i\le r-1}\\
=&\gwTr{1}{1,\ell}\prod_{i=1}^{r-1}\frac{\gwTr{n-1}{1,k_i}}{\gwTr{n}{1,1}}
\sum_{\ell=r}^\infty
\sum_{j=r-1}^{\ell-1}\binom{j-1}{r-2}
\gwq_{n-1}^{j-r+1}\frac{\gwTr{n-1}{\ell-j,k_r}}{\gwTr{n}{1,1}}.
\end{aligned}
\end{equation}
By \Cref{pp:ratio} and dominated convergence, this term converges to
\begin{align*}
&\gwTr{1}{1,\ell}\gwgamma^{-r}\prod_{i=1}^{r-1}\pi_{k_i}
\sum_{\ell=r}^\infty
\sum_{j=r-1}^{\ell-1}\binom{j-1}{r-2}
\gwq^{j-r+1+\ell-j-1}(\ell-j)\pi_{k_r}\\
=&\gwTr{1}{1,\ell}\gwgamma^{-r}\prod_{i=1}^{r}\pi_{k_i}
\sum_{\ell=r}^\infty\binom{\ell}{r}
\gwq^{\ell-r}\\
=&\gwgamma^{-r}\frac{f^{(r)}(q)}{r!}\prod_{i=1}^{r}\pi_{k_i},
\end{align*}
by the elementary identities 
\[
\sum_{i=r-1}^{\ell-1}(\ell-i)\binom{i-1}{r-2}=\binom{\ell}{r},
\,
\gwf^{(r)}(s)=\sum_{\ell\ge r}r!\binom{\ell}{r}\gwTr{1}{1,\ell} s^{\ell-r}. 
\]
Moreover, by \Cref{pp:ratio} again, we have that
\[
\frac{\gwTr{n}{1,k}}{(\gwTr{n}{1,1})^r}=(1+o(1))\frac{\pi_k}{(\gwTr{n}{1,1})^{r-1}},
\]
thus \eqref{eq:pgw2} gives
\begin{equation*}
\begin{aligned}
&\lim_{n\rightarrow\infty}\frac{\pi_k}{(\gwTr{n}{1,1})^{r-1}}\pruP{n,k}\pars*{\gnZ{1}(T)\ge r,\gnZ{n-1}(\subtree{\gnT}{i})=k_i, 1\le i\le r-1}\\
=&\gwgamma^{-r}\frac{f^{(r)}(q)}{r!}\pi_{k_1}\cdots\pi_{k_r},
\end{aligned}
\end{equation*}
and the result follows by summing over all choices of $(k_i).$
\end{proof}
\begin{rmk}\label{assumption_derivative}
The condition $\gwf^{(r)}(s)<\infty$ is always true if $s<1$, while $\gwq<1$ if and only if $\mean\le 1$. Thus the condition $\gwf^{(r)}(\gwq)<\infty$ in \Cref{cor:pru_r} is trivially satisfied if $\mean>1$. If $\mean\le 1$, one can verify that this condition is equivalent to the moment condition $\sum_{i=1}^\infty i^r\gwCh(i)<\infty$.
\end{rmk}
\subsection{The small tree measure}
Now we study the composition of the cut and prune operation.
\begin{defn}\label{defn:st}
For any pair of integers $(n,k)$ such that $\gwTr{n}{1,k}>0$ and $k\ge 2$,
we denote by $\stP{n,k}$ the law of $\cut{n}(\gnT)$, where $\gnT$ is sampled under the law of $\pruP{n,k}$. 
In other words, for any $A\subseteq\gnTrees$,
\[
\stP{n,k}(A)=\pruP{n,k}(\cut{n}(\gnT)\in A).
\]
\end{defn}
We remark that $k$ is assumed to be at least $2$, since for $k=1$, the cut operation will always give the trivial tree $\braces{\varnothing}$. This measure $\stP{n,k}$ is developed in the following lemma:
\begin{lem}\label{lem:st}
Under \eqref{eq:assumption_gw}, for any pair of integers $(n,k)$ such that $\gwTr{n}{1,k}>0$ and $k\ge 2$,
$\stP{n,k}$ is supported on $\cup_{u=1}^n\pruTrees{u,k}\cap\setof{\gnT}{\gnZ{1}\ge 2}$. 
Moreover, for any $A\subseteq \cup_{u=1}^n\pruTrees{u,k}\cap\setof{\gnT}{\gnZ{1}\ge 2}$,
\begin{equation*}
\begin{aligned}
\stP{n,k}(A)
=\frac{\gwTr{n}{1,1}}{\gwTr{n}{1,k}}\sum_{u=1}^n\frac{\gwTr{u}{1,k}}{\gwTr{u}{1,1}}\pruP{u,k}(A\cap\pruTrees{u,k}).
\end{aligned}
\end{equation*}
\end{lem}
\begin{proof}
For the first assertion, by construction, $\cut{n}(\gnT)$ has at least $2$ nodes in its first generation. Moreover, the trees on which we perform the cut operation are those in $\pruTrees{n,k}$, so $\cut{n}(\gnT)$ has height at most $n$, with all its leaves in the last generation. Thus it is an element in $\cup_{u=1}^n\pruTrees{u,k}\cap\setof{\gnT}{\gnZ{1}(\gnT)\ge 2}$. The converse is trivial.

Then it suffices to prove the second assertion for $A\subseteq \pruTrees{u,k}\cap\setof{\gnT}{\gnZ{1}(\gnT)\ge 2}.$ Indeed, by \Cref{defn:st} and \Cref{cor:pruHeight},
\begin{equation*}
\begin{aligned}
\stP{n,k}(A)
&=\pruP{n,k}(\cut{n}(\gnT)\in A)\\
&=\pruP{n,k}(\gnZ{1}(\gnT)=\cdots=\gnZ{n-u}(\gnT)=1,\subtree{\gnT}
{\underbrace{11\dotsb1}_{\text{$n-u$ times}}}\in A)\\
&=\frac{\gwTr{n}{1,1}}{\gwTr{n}{1,k}}
\frac{\gwTr{u}{1,k}}{\gwTr{u}{1,1}}
\pruP{u,k}(A).
\end{aligned}
\end{equation*}
Then the conclusion follows by partitioning a general set $A\subseteq \cup_{u=1}^n\pruTrees{u,k}\cap\setof{\gnT}{\gnZ{1}\ge 2}$ into $\cup_{u=1}^n (A\cap\pruTrees{u,k})$, and use the above equation on each part $A\cap\pruTrees{u,k}$.
\end{proof}
In fact, we notice that $n$ no longer plays a major role in \Cref{lem:st}, and we are motivated to take the limit $n\rightarrow\infty$. Moreover, this limit measure still has Galton-Watson-type branching properties:
\begin{prop}\label{pp:st}
Under \eqref{eq:assumption_gw}, fix $k\ge 2$, let $n\rightarrow\infty$, then the measures $\pars{\stP{n,k}}_n$ converge to a measure $\stP{k}$ supported on $\cup_{u=1}^\infty\pruTrees{u,k}\cap\setof{\gnT}{\gnZ{1}(\gnT)\ge 2}$, defined by
\begin{equation}\label{eq:st1}
\stP{k}(A)
=\frac{1}{\pi_k}\sum_{u=1}^\infty
\frac{\gwTr{u}{1,k}}{\gwTr{u}{1,1}}\pruP{u,k}(A\cap\pruTrees{u,k}),
\end{equation}
for any $A\subseteq\cup_{u=1}^\infty\pruTrees{u,k}\cap\setof{\gnT}{\gnZ{1}(\gnT)\ge 2}$.
Moreover, fix any $u>0$ such that $\gwTr{u}{1,k}>0$, let $r\ge 2$ and $k_1,\cdots,k_r\ge 1$ such that $\sum_{i=1}^r k_i=k$. Then for any $A_i\subseteq\pruTrees{u-1,k_i},$
\begin{equation}\label{eq:st2}
\begin{aligned}
&\stP{k}\pars*{\gnZ{1}(T)=r,\subtree{\gnT}{i}\in A_i, 1\le i\le r}\\
=&
\frac{\pi_{k_1}\cdots\pi_{k_r}}{\pi_k}
\frac{\gwTr{u-1}{1, 1}^{r}}{\gwTr{u}{1, 1}}
\frac{\gwf^{(r)}(\gwq_{u-1})}{r!}
\prod_{i=1}^r \stP{k_i}(\cut{u-1}(A_i)).
\end{aligned}
\end{equation}
\end{prop}
\begin{proof}
Convergence and \eqref{eq:st1} follows directly from \Cref{lem:st} and Part 1 of \Cref{pp:ratio}.

Then for \eqref{eq:st2}, by \eqref{eq:st1} and \Cref{prop:condi_is_multi}, we have that
\begin{equation}\label{eq:st3}
\begin{aligned}
&\stP{k}(\gnZ{1}(T)=r,\subtree{\gnT}{i}\in A_i, 1\le i\le r)\\
=&\frac{1}{\pi_k}
\frac{\gwTr{u}{1,k}}{\gwTr{u}{1,1}}
\pruP{u,k}(\gnZ{1}(T)=r,\subtree{\gnT}{i}\in A_i, 1\le i\le r)\\
=&\frac{1}{\pi_k}
\frac{1}{\gwTr{u}{1,1}}
\frac{\gwf^{(r)}(\gwq_{u-1})}{r!}
\prod_{i=1}^r \gwTr{u-1}{1,k_i}
\prod_{i=1}^r \pruP{u-1,k_i}(A_i).
\end{aligned}
\end{equation}
Decompose $A_i$ by the height of trees after the cut operation, 
\[A_i^{(x)}:=\setof*{\gnT\in A_i}{\gnH(\cut{u-1}(\gnT))=x},\]
then by \Cref{cor:pruHeight},
\begin{align*}
\pruP{u-1,k_i}(A_i)=
\frac{\gwTr{u-1}{1,1}}{\gwTr{u-1}{1,k_i}}
\sum_{x=1}^{u-1}
\frac{\gwTr{x}{1,k_i}}{\gwTr{x}{1,1}}
\pruP{x,k_i}(\cut{u-1}(A_i^{(x)})).
\end{align*}
Since $\cut{u-1}$ is injective on $A_i$, we have 
\[\cut{u-1}(A_i^{(x)})=\cut{u-1}(A_i)\cap\pruTrees{x,k_i},\]
thus by \eqref{eq:st1} again, 
\begin{align*}
\pruP{u-1,k_i}(A_i)=
\frac{\gwTr{u-1}{1,1}}{\gwTr{u-1}{1,k_i}}
\cdot\pi_{k_i}
\stP{k_i}(\cut{u-1}(A_i)).
\end{align*}
Put this back into \eqref{eq:st3}, and we get \eqref{eq:st2}.
\end{proof}

This enables us to give further descriptions of $\stP{k}$. Recall that $\gnH(\gnT)$ is the height of a tree, then
\begin{cor}\label{cor:branching_st}
Under \eqref{eq:assumption_gw}, for any $r,k\ge 2$, if $\gwf^{(r)}(\gwq)<\infty$, then
\begin{align*}
&\lim_{u\rightarrow\infty}
\frac{\stP{k}\parsof*{\gnZ{1}(\gnT)=r}{\gnH(\gnT)=u}}{\gwTr{u}{1,1}^{r-2}}\\
=&\lim_{u\rightarrow\infty}
\frac{\stP{k}\parsof*{\gnZ{1}(\gnT)\ge r}{\gnH(\gnT)=u}}{\gwTr{u}{1,1}^{r-2}}
=\frac{2\gwgamma^{2-r}}{r!}\frac{\gwf^{(r)}(\gwq)}{\gwf''(\gwq)}
\frac{\sum_{\substack{k_1, k_2, \cdots, k_r \ge 1\\ k_1 + \dotsb + k_r = k}}
\pi_{k_1} \dotsb \pi_{k_r}}
{\sum_{\substack{k_1, k_2\ge 1\\ k_1 + k_2 = k}}
\pi_{k_1} \pi_{k_2}}.
\end{align*}
\end{cor}
\begin{proof}
By \eqref{eq:st1},
\begin{align*}
\stP{k}\parsof*{\gnZ{1}(\gnT)=r}{\gnH(\gnT)=u}=
\frac{\pruP{u,k}({\gnZ{1}(\gnT)=r})}
{\pruP{u,k}({\gnZ{1}(\gnT)\ge 2})},
\end{align*}
and we apply \Cref{cor:pru_r}. Changing $\gnZ{1}(\gnT)=r$ to $\gnZ{1}(\gnT)\ge r$ is idem.
\end{proof}

\begin{cor}\label{cor:st_height}

Under \eqref{eq:assumption_gw}, fix any $k\ge 2$.
\begin{enumerate}
\item For any $u\ge 1$,
\[
\stP{k}(\gnH(\gnT)=u)
=\frac{1}{\pi_k}\pars*{
\frac{\gwTr{u}{1,k}}{\gwTr{u}{1,1}}-\frac{\gwTr{u-1}{1,k}}{\gwTr{u-1}{1,1}}}.
\]
\item 
\begin{align*}
\lim_{u\rightarrow\infty}\frac{\stP{k}(\gnH(\gnT)=u)}{\gwTr{u}{1,1}}
=
\gwgamma^{-2} \frac{\gwf''(\gwq)}{2}
\frac{\sum_{1\le i\le k-1}
\pi_{i}\pi_{k-i}}{\pi_k}
.
\end{align*}
\end{enumerate}
\end{cor}
\begin{proof}
\begin{enumerate}
\item
Take $A=\pruTrees{u,k}\cap\setof{\gnT}{\gnZ{1}(\gnT)\ge 2}$ in \Cref{pp:st}, we have that
\begin{align*}
&\stP{k}(\gnH(\gnT)=u)\\
=&\stP{k}(\pruTrees{u,k}\cap\setof{\gnT}{\gnZ{1}(\gnT)\ge 2})\\
=&\frac{1}{\pi_k}\frac{\gwTr{u}{1,k}}{\gwTr{u}{1,1}}
\pruP{u,k}({\gnZ{1}(\gnT)\ge 2})\\
=&\frac{1}{\pi_k}\frac{\gwTr{u}{1,k}}{\gwTr{u}{1,1}}
\pars*{1-\pruP{u,k}({\gnZ{1}(\gnT)=1})},
\end{align*}
then we use \Cref{cor:pruHeight} to conclude that
\begin{align*}
\frac{1}{\pi_k}\frac{\gwTr{u}{1,k}}{\gwTr{u}{1,1}}
\pars*{1-\pruP{u,k}({\gnZ{1}(\gnT)=1})}
=&\frac{1}{\pi_k}\frac{\gwTr{u}{1,k}}{\gwTr{u}{1,1}}
\pars*{1-\frac{\gwTr{u-1}{1,k}}{\gwTr{u-1}{1,1}}\frac{\gwTr{u}{1,1}}{\gwTr{u}{1,k}}}\\
=&\frac{1}{\pi_k}
\pars*{\frac{\gwTr{u}{1,k}}{\gwTr{u}{1,1}}-\frac{\gwTr{u-1}{1,k}}{\gwTr{u-1}{1,1}}}.
\end{align*}
\item
In the proof of Part 1, we deduced that
\[
\stP{k}(\gnH(\gnT)=u)
=\frac{1}{\pi_k}\frac{\gwTr{u}{1,k}}{\gwTr{u}{1,1}}
\pruP{u,k}({\gnZ{1}(\gnT)\ge 2}),
\]
and the conclusion follows from \Cref{cor:pru_r} with $r=2$. 
\end{enumerate}
\end{proof}
\begin{rmk}\label{rmk:ratio}
As a byproduct of \Cref{cor:st_height}, 
we have that
\[
\frac{1}{\pi_k}\pars*{
\frac{\gwTr{u}{1,k}}{\gwTr{u}{1,1}}-\frac{\gwTr{u-1}{1,k}}{\gwTr{u-1}{1,1}}}
=(1+o(1))\gwgamma^{-2} \frac{\gwf''(\gwq)}{2}
\frac{\sum_{1\le i\le k-1}
\pi_{i}\pi_{k-i}}{\pi_k}
\gwTr{u}{1,1}.
\]
Together with the asymptotic of $\gwTr{u}{1,1}$ in \Cref{pp:ratio}, we deduce that
\[
\frac{\gwTr{u}{1,k}}{\gwTr{u}{1,1}}-\frac{\gwTr{u-1}{1,k}}{\gwTr{u-1}{1,1}}
=
\begin{cases}
(C_4+o(1))u^{-2},        & \mean=1, \\
(C_5+o(1))\gwgamma^{u},  & \mean\ne 1,
\end{cases}
\]
where
\begin{align*}
C_4=
\gwgamma^{-2} \frac{\gwf''(\gwq)\sum_{1\le i\le k-1}\pi_{i}\pi_{k-i}}
{\sigma^2\sum_{i=1}^\infty \pi_i (\gwCh(0))^i},\,
C_5=\frac{1}{2}\gwgamma^{-2} v_1{\gwf''(\gwq)\sum_{1\le i\le k-1}\pi_{i}\pi_{k-i}},
\end{align*}
with $\sigma^2,v_1$ defined in \Cref{pp:ratio}.
\end{rmk}

\section{Application to branching random walks}\label{sec:spatial}
As we shall deal with trees without fixed heights in this section, we abbreviate $\range{}(\gnT)=\range{\gnH(\gnT)}(\gnT)$ and $\gap{}{i}(\gnT)=\gap{\gnH(\gnT)}{i}(\gnT)$.
By \eqref{eq:operation_nature} and \Cref{pp:st}, for the span we have that
\begin{equation}\label{eq:PGWtoPST}
\begin{aligned}
&\lim_{n\rightarrow\infty}\bgwP\parsof*{\range{n}>x}{\gnZ{n}=k}\\
=&\lim_{n\rightarrow\infty}\bgwP\parsof*{\range{}(\cut{n}(\prune{n}\gnsT))>x}{\gnZ{n}=k}\\
=&\lim_{n\rightarrow\infty}\bstP{n,k}\pars*{\range{}(\gnsT)>x}=\bstP{k}\pars*{\range{}(\gnsT)>x},
\end{aligned}
\end{equation}
and it is idem for the gaps. In other words,
$\range{n}$ and $(\gap{n}{i})$ under $\bgwP\parsof{\cdot}{\gnZ{n}=k}$ converge to $\range{}(\gnsT)$ and $(\gap{}{i}(\gnsT))$ under $\bstP{k}(\cdot)$ as $n\rightarrow\infty$.
Thus to prove \Cref{thm:brw} it suffices to study the span and gaps under $\bstP{k}$.
\subsection{The span}
Take any tree $\gnsT\in\pruTrees{n,k}\cap\braces{\gnZ{1}(\gnsT)\ge 2}$ under $\bstP{k}$, 
we divide the span $\range{n}(\gnsT)$ into two parts: 
the span of the first (in lexicographical order) node in the last layer of each subtree 
is denoted by 
\[\rangeS{n}(\gnsT):=\text{the span of }\setof{1\le i\le\gnZ{1}(\gnsT)}{\gnV{\underbrace{i11\dotsb1}_{\text{length $n$}}}(\gnsT)},\]
and the maximum span among each subtree is denoted by
\[\rangeG{n}(\gnsT):=\max_{1\le i\le \gnZ{1}(\gnsT)}\braces{\range{n-1}(\subtree{\gnsT}{i})}.\] 

\begin{figure}[ht]
\centering
\includegraphics[scale=0.8]{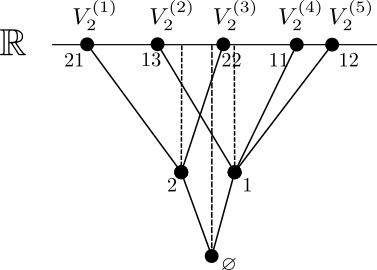}
\caption{A spatial tree.}
\label{fig3}
\end{figure}
For instance, the tree in \Cref{fig3} has height $n=2$ and $k=5$ particles in the last generation, with two subtrees at the first generation. For this tree we have  
$\rangeS{2}(\gnsT)=\gnV{2}^{(4)}(\gnsT)-\gnV{2}^{(1)}(\gnsT)$ 
and
$\rangeG{2}(\gnsT)=\max\braces{\gnV{2}^{(3)}(\gnsT)-\gnV{2}^{(1)}(\gnsT),\gnV{2}^{(5)}(\gnsT)-\gnV{2}^{(2)}(\gnsT)}$.

By the triangle inequality, we have that
\begin{equation}\label{eq:rangeRSG}
\rangeS{n}(\gnsT)\le\range{n}(\gnsT)\le\rangeS{n}(\gnsT)+2\rangeG{n}(\gnsT).
\end{equation}

For simplicity, since trees under $\bstP{k}$ do not have a fixed height, 
we write $\rangeS{}(\gnsT),\,\rangeG{}(\gnsT),\,\range{}(\gnsT)$ for $\rangeS{\gnH(\gnsT)}(\gnsT),\,\rangeG{\gnH(\gnsT)}(\gnsT),\,\range{\gnH(\gnsT)}(\gnsT)$.

\begin{lem}\label{lm:rangeS}
Fix $k\ge 2$. Under the conditions \eqref{eq:assumption_gw} and \eqref{eq:assumption_brw}, 
as $x\rightarrow\infty$, 
\[
\bstP{k}\pars*{\rangeS{}(\gnsT)>x}=
\begin{cases}
(C_1+o(1))x^{-2}, &\mean=1,\\
\exp(-(C_2+o(1))x),&\mean\ne 1,
\end{cases}
\]
where
\[
C_1=
\frac{\gwf''(\gwq)}{\gwgamma^2\sigma^2}
\frac{\sum_{1\le i\le k-1}
\pi_{i}\pi_{k-i}}{\pi_k\sum_{i=1}^\infty \pi_i (\gwCh(0))^i},
\]
\[
C_2=\inf_{s\in(0,\infty)}\pars*{-s\log \gamma+\sup_{t\in\mathbb R}(t-s\Lambda(t))},
\]
and all parameters appearing in the $C_1,C_2$ are those in \Cref{pp:ratio} and \Cref{lm:cremer}.
\end{lem}
\begin{proof}
Denote by $\gnRV_i^{(j)}$ independent random variables distributed as $\brwD$ for all $i,j\in\mathbb N$. Denote 
\[
\rangeR(m,u,x):=\gnP\pars*{
\max_{1\le a,b\le m}\braces*{\sum_{i=1}^u \gnRV_i^{(a)}-\sum_{j=1}^u\gnRV_j^{(b)}}
>x},
\]
then
\begin{align*}
\bstP{k}(\rangeS{}(\gnsT)>x)&=\bstE{k}(\rangeR(\gnZ{1}(\gnsT),\gnH(\gnsT),x))\\
&=\sum_{u\ge 1}\bstP{k}(\gnH(\gnsT)=u)\bstE{k}\parsof*{\rangeR(\gnZ{1}(\gnsT),u,x)}{\gnH(\gnsT)=u}.
\end{align*}
Moreover, by the union bound, 
\[
\rangeR(m,u,x)\le m^2\rangeR(2,u,x),
\]
thus
\begin{equation}\label{eq:lem41}
\begin{aligned}
&\bstE{k}(\rangeR(\gnZ{1}(\gnsT),\gnH(\gnsT),x))
-\sum_{u\ge 1}\bstP{k}(\gnH(\gnsT)=u)\rangeR(2,u,x)\\
\le&\sum_{u\ge 1}\bstP{k}(\gnH(\gnsT)=u)
\cdot{k^2\rangeR(2,u,x)\bstP{k}\parsof*{\gnZ{1}(\gnsT)\ge 3}{\gnH(\gnsT)=u}}.
\end{aligned}\end{equation}
Then by \Cref{cor:branching_st}, the error term in \eqref{eq:lem41} is negligible, so
\[
\bstP{k}(\rangeS{}(\gnsT)>x)=\bstE{k}(\rangeR(\gnZ{1}(\gnsT),\gnH(\gnsT),x))
=(1+o(1))\sum_{u\ge 1}\bstP{k}(\gnH(\gnsT)=u)\rangeR(2,u,x).
\]
Thus it suffices to show that 
\[
\sum_{u\ge 1}\bstP{k}(\gnH(\gnsT)=u)\rangeR(2,u,x)=
\begin{cases}
(C_1+o(1))x^{-2}, &\mean=1,\\
\exp(-(C_2+o(1))x),&\mean\ne 1.
\end{cases}
\]

If $\mean=1$, by Part 2 of \Cref{cor:st_height}, Part 3 of \Cref{pp:ratio} and Part 1 of \Cref{lm:cremer}, we have that
\begin{align*}
&\sum_{u\ge 1}\bstP{k}(\gnH(\gnsT)=u)\rangeR(2,u,x)\\
=&\sum_{u>x^{3/2}}\frac{C_1+o(1)}{u^2}
\left[2\left(1-\Phi\left(\frac{x}{\sqrt {2u}}\right)\right)
+O\left(\frac{1}{\sqrt u}\right)\right]
 +{\sum_{u\le x^{3/2}}o\pars*{\frac{1}{u^2}\cdot\frac{u}{x^2\log x}}}\\
=&(2C_1+o(1))\int_{x^{3/2}}^\infty\frac{1-\Phi(x/\sqrt{2y})}{y^2}dy+o(x^{-2})\\
=&\frac{2C_1+o(1)}{x^2}\int_0^\infty\frac{1-\Phi(1/\sqrt{2z})}{z^2}dz=\frac{2C_1+o(1)}{x^2}.
\end{align*}

If $\mean\ne1$, similarly,
we can choose suitable constants $C,s_1,s_2$ such that
\begin{align*}
&\sum_{u\ge 1}\bstP{k}\pars{\gnH(\gnsT)=u}\rangeR(2,u,x)\\
=&\sum_{s_1x<u<s_2x}{(C+o(1))}{\gamma^{u}}
{e^{-(1+o(1))\sup_{t\in\mathbb R}(tx-u\Lambda(t))}}
+{O\pars*{R(2,s_1x,x)+\gamma^{s_2x}}}\\
=&e^{-(C_2+o(1))x}.
\end{align*}

We remark that since
\begin{align*}
\lim_{s\rightarrow+\infty}\pars*{-s\log \gamma+\sup_{t\in\mathbb R}(t-s\Lambda(t))}
&\ge \lim_{s\rightarrow+\infty}\pars*{-s\log \gamma+(0-s\Lambda(0))}\\
&=\lim_{s\rightarrow+\infty}\pars*{-s\log \gamma}=+\infty,
\end{align*}
and 
\begin{align*}
\lim_{s\rightarrow0+}\pars*{-s\log \gamma+\sup_{t\in\mathbb R}(t-s\Lambda(t))}
&= \lim_{s\rightarrow0+}\sup_{t\in\mathbb R}(t-s\Lambda(t))=+\infty,
\end{align*}
the infimum over $(0,\infty)$ in $C_2$ is equivalent to the infimum among a bounded interval $[\epsilon,\epsilon^{-1}]$.
\end{proof}
\begin{prop}\label{pp:rangeR}
Fix $k\ge 2$. Under the conditions \eqref{eq:assumption_gw} and \eqref{eq:assumption_brw}, 
as $x\rightarrow\infty$, 
\[
\bstP{k}\pars*{\range{}(\gnsT)>x}=
\begin{cases}
(C_1+o(1))x^{-2}, &\mean=1,\\
\exp(-(C_2+o(1))x),&\mean\ne 1,
\end{cases}
\]
where $C_1,C_2$ are those in \Cref{lm:rangeS}.
\end{prop}
\begin{proof}
By \eqref{eq:rangeRSG} and \Cref{lm:rangeS}, it suffices to show that
\begin{equation}\label{eq:rangeGoS}
\bstP{k}(\rangeG{}(\gnsT)>x^{1-\epsilon})=o(\bstP{k}(\rangeS{}(\gnsT)>x))
\end{equation}
for some $\epsilon>0$.

In fact, $\rangeG{}(\gnsT)$ is determined by the structures of $\cut{}(\subtree{\gnsT}{i})$.
If we denote by
\[
\tilde{\gnH}(\gnsT):=\max_{1\le i\le\gnZ{1}(\gnsT)}\gnH(\cut{}(\subtree{\gnsT}{i})),
\]
then by the union bound, $\rangeG{}(\gnsT)$ is determined by 
at most $k^2$ pairs of nodes within the same subtrees, in other words,
\begin{align*}
\bstP{k}(\rangeG{}(\gnsT)>x^{1-\epsilon}) 
&\le 
k^2\bstE{k}\bracks*{\max_{1\le u\le\tilde{\gnH}(\gnsT)}\rangeR(2,u,x^{1-\epsilon})}\\
&\asymp \sum_{u\ge 1}\bstP{k}(\tilde{\gnH}(\gnsT)=u)\rangeR(2,u,x^{1-\epsilon}).
\end{align*}
Sum over all possible cases by \eqref{eq:st2}, notice that $r$ and $(k_i)$ can only take integer values at most $k$, we have that 
\[
\bstP{k}(\tilde{\gnH}(\gnsT)=u)\lesssim \gwTr{u}{1,1}\bstP{k}(\gnH(\gnsT)=u).
\]
To sum up,
\[
\bstP{k}(\rangeG{}(\gnsT)>x^{1-\epsilon}) 
\lesssim \sum_{u\ge 1}\gwTr{u}{1,1}\bstP{k}(\gnH(\gnsT)=u)\rangeR(2,u,x^{1-\epsilon}),
\]
while
\[
\bstP{k}(\rangeS{}(\gnsT)>x) 
\asymp \sum_{u\ge 1}\bstP{k}(\gnH(\gnsT)=u)\rangeR(2,u,x).
\]
Therefore, by the proof of \Cref{lm:rangeS}, we have the desired dominance in \eqref{eq:rangeGoS}.
\end{proof}
\subsection{The gaps}
\begin{prop}\label{pp:gap}
Let $k\ge 2$ and $1\le i\le k-1$. Under \eqref{eq:assumption_gw} and \eqref{eq:assumption_brw}, as $x\rightarrow\infty$,
\begin{align*}
\bstP{k}(\gap{}{i}(\gnsT)>x)
&=(C_3+o(1))\bstP{k}(\range{}(\gnsT)>x)\\
&=
\begin{cases}
(C_1C_3+o(1))x^{-2}, &\mean=1,\\
\exp(-(C_2+o(1))x),&\mean\ne 1,
\end{cases}
\end{align*}
where $C_3=\frac{\pi_i\pi_{k-i}}{\sum_{j=1}^{k-1}\pi_j\pi_{k-j}}$.
\end{prop}
\begin{proof}
Recall that trees under $\bstP{k}$ have $k$ nodes in the last generation, and we write their positions in increasing order,
\[\gnV{}^{(1)}(\gnsT)\le\cdots\le\gnV{}^{(k)}(\gnsT).\]
In the previous section, we showed that 
\begin{align*}
&\bstP{k}\pars*{\gnZ{1}(\gnsT)\ge 3,\range{}(\gnsT)>x}=o(\bstP{k}(\range{}(\gnsT)>x)),\\
&\bstP{k}\pars*{\rangeG{}(\gnsT)>x^{1-\epsilon}}=o(\bstP{k}(\range{}(\gnsT)>x)),
\end{align*}
thus it suffices to consider the case where
\[
\braces{\gnV{}^{(1)}(\gnsT),\cdots,\gnV{}^{(i)}(\gnsT)}\text{ and }
\braces{\gnV{}^{(i+1)}(\gnsT),\cdots,\gnV{}^{(k)}(\gnsT)}
\]
are exactly positions of nodes in the last generation of the two subtrees $\subtree{\gnsT}{1},\subtree{\gnsT}{2}$.
In other words,
\begin{align*}
 &\bstP{k}(\gap{}{i}(\gnsT)>x)\\
=&\frac{1}{2}\bstP{k}(\rangeS{}(\gnsT)>x,\#\subtree{\gnsT}{1}=i,\#\subtree{\gnsT}{2}=k-i)\\
&+\frac{1}{2}\bstP{k}(\rangeS{}(\gnsT)>x,\#\subtree{\gnsT}{1}=k-i,\#\subtree{\gnsT}{2}=i)+o(\bstP{k}(\range{}(\gnsT)>x)),
\end{align*}
where the factors $\frac{1}{2}$ are to distinguish the symmetric cases $\gap{}{i}(\gnsT)>x$ and $\gap{}{k-i}(\gnsT)>x$.

Moreover, by \Cref{pp:st}, we have that
\begin{align*}
&\bstP{k}(\rangeS{}(\gnsT)>x,\#\subtree{\gnsT}{1}=i,\#\subtree{\gnsT}{2}=k-i)\\
=&\sum_{u\ge 1}\rangeR(2,u,x)
\bstP{k}(\gnH(\gnsT)=u,\#\subtree{\gnsT}{1}=i,\#\subtree{\gnsT}{2}=k-i)\\
=&(C_3+o(1))\sum_{u\ge 1}\rangeR(2,u,x)
\bstP{k}(\gnH(\gnsT)=u),
\end{align*}
and the conclusion follows from the proof of \Cref{lm:rangeS}.
\end{proof}
\begin{rmk}\label{rmk:phy}
As an example, consider the canonical case where the offspring distribution $\gwCh$ is the geometric distribution with parameter $\frac{1}{2}$, i.e. $\gwCh(k)=2^{-k-1}$, then one can explicitly show (cf. eg. \cite[Section 1.4]{bookan}) that its generating function satisfies
\[
f_n(s)=\sum_{i=0}^\infty\mathbb P(Z_n=i)s^i=1-\frac{1}{n+\frac{1}{1-s}},\]
and its transition probabilities are
\[P_n(1,k)=\frac{1}{k!}\frac{d^kf_n(s)}{ds^k}|_{s=0}=\frac{n^{k-1}}{(n+1)^{k+1}}.\]
Therefore,
\[
\pi_k=\lim_{n\rightarrow\infty}\frac{P_n(1,k)}{P_n(1,1)}=1,\, \forall k\in\mathbb N_+.
\]
Thus in this case, the constant 
\[C_1C_3=\frac{\gwf''(\gwq)}{\gwgamma^2\sigma^2}
\frac{\pi_i\pi_{k-i}}{\pi_k\sum_{i=1}^\infty \pi_i (\gwCh(0))^i}=\frac{\gwf''(\gwq)}{\gwgamma^2\sigma^2}
\frac{1}{\sum_{i=1}^\infty (\gwCh(0))^i}
\]
in \Cref{pp:gap} no longer depends on the choice of $i$ or $k$, as is showed in \cite{phycon2}.
\end{rmk}

Finally, we formally conclude that
\begin{proof}[Proof of \Cref{thm:brw}]
The conclusion follows from \eqref{eq:PGWtoPST} (with its counterpart for the gaps), \Cref{pp:rangeR} and \Cref{pp:gap}.
\end{proof}
\bibliography{1}

\begin{thebibliography}{10}

\bibitem{zbMATH07277656}
Romain {Abraham}, Aymen {Bouaziz}, and Jean-Fran\c{c}ois {Delmas}.
\newblock {Very fat geometric Galton-Watson trees}.
\newblock {\em {ESAIM, Probab. Stat.}}, 24:294--314, 2020.

\bibitem{zbMATH06347442}
Romain {Abraham} and Jean-Fran\c{c}ois {Delmas}.
\newblock {Local limits of conditioned Galton-Watson trees: the condensation
  case}.
\newblock {\em {Electron. J. Probab.}}, 19:29, 2014.
\newblock Id/No 56.

\bibitem{zbMATH06257619}
Romain {Abraham} and Jean-Fran\c{c}ois {Delmas}.
\newblock {Local limits of conditioned Galton-Watson trees: the infinite spine
  case}.
\newblock {\em {Electron. J. Probab.}}, 19:19, 2014.
\newblock Id/No 2.

\bibitem{bookan}
Krishna~B. Athreya and Peter~E. Ney.
\newblock {\em Branching processes}.
\newblock Dover Publications, Inc., Mineola, NY, 2004.

\bibitem{zbMATH06696268}
Nicolas {Curien} and Jean-Fran\c{c}ois {Le Gall}.
\newblock {The harmonic measure of balls in random trees}.
\newblock {\em {Ann. Probab.}}, 45(1):147--209, 2017.

\bibitem{fleischmann1974grenzwertsatz}
Klaus Fleischmann and Uwe Prehn.
\newblock Ein grenzwertsatz f{\"u}r subkritische verzweigungsprozesse mit
  endlich vielen typen von teilchen.
\newblock {\em {Math. Nachr.}}, 64(1):357--362, 1974.

\bibitem{zbMATH03469834}
Klaus {Fleischmann} and Rainer {Siegmund-Schultze}.
\newblock {The structure of reduced critical Galton-Watson processes}.
\newblock {\em {Math. Nachr.}}, 79:233--241, 1971.

\bibitem{geiger_1999}
Jochen Geiger.
\newblock Elementary new proofs of classical limit theorems for
  {G}alton-{W}atson processes.
\newblock {\em J. Appl. Probab}, 36(2):301--309, 1999.

\bibitem{zbMATH04028592}
Harry {Kesten}.
\newblock {Subdiffusive behavior of random walk on a random cluster}.
\newblock {\em {Ann. Inst. Henri Poincar\'e, Probab. Stat.}}, 22:425--487,
  1986.

\bibitem{zbMATH07014723}
Minzhi {Liu} and Vladimir~A. {Vatutin}.
\newblock {Reduced critical branching processes for small populations}.
\newblock {\em {Theory Probab. Appl.}}, 63(4):648--656, 2019.

\bibitem{petrov1995limit}
Valentin~V. Petrov.
\newblock {\em Limit theorems of probability theory: sequences of independent
  random variables}.
\newblock Oxford University Press, 1995.

\bibitem{phycon}
Kabir Ramola, Satya~N. Majumdar, and Gr{\'e}gory Schehr.
\newblock Universal order and gap statistics of critical branching brownian
  motion.
\newblock {\em Phys. Rev. Lett.}, 112(21):210602, 2014.

\bibitem{phycon2}
Kabir {Ramola}, Satya~N. {Majumdar}, and Gr\'egory {Schehr}.
\newblock {Branching Brownian motion conditioned on particle numbers}.
\newblock {\em {Chaos Solitons Fractals}}, 74:79--88, 2015.

\end{thebibliography}
\end{document}